\documentclass[11pt]{amsart}
\usepackage{amsmath} 
\usepackage{amsthm}
\usepackage{amssymb} 
\usepackage{amsfonts}
\usepackage{graphicx} 
\usepackage{inputenc}
\usepackage{upref} 
\usepackage{bm}
\usepackage{calc}
\usepackage{enumerate}
\usepackage[none]{hyphenat}
\usepackage{caption}
\usepackage{hyperref}
\usepackage{color}
\usepackage{multicol}
\usepackage{subcaption}
\usepackage{thmtools}
\usepackage{float}
\usepackage{subcaption}
\usepackage{thm-restate}

\newtheorem{theorem}{Theorem}[section]
\newtheorem{conjecture}[theorem]{Conjecture}
\newtheorem{lemma}[theorem]{Lemma}

\newtheorem{proposition}[theorem]{Proposition}
\newtheorem{corollary}[theorem]{Corollary}

\newtheorem{exisreg}[theorem]{Existence and Regularity}
\newtheorem{firstvar}[theorem]{First Variation Formulas}
\newtheorem{secondV_0ar}[theorem]{Second Variation Formula}

\theoremstyle{definition}
\newtheorem{definition}[theorem]{Definition}
\newtheorem{DGL}[theorem]{The Double Gaussian Line}
\newtheorem{DGP}[theorem]{The Double Gaussian Plane}

\renewcommand{\phi}{\varphi}

\DeclareMathOperator{\sech}{sech}
\DeclareMathOperator{\arccosh}{arccosh}

\graphicspath{{./Figures/}}

\title{The isoperimetric problem in the plane with the sum of two Gaussian densities}
\author[J. Berry]{John Berry}
\author[M. Dannenberg]{Matthew Dannenberg}
\author[J. Liang]{Jason Liang}
\author[Y. Zeng]{Yingyi Zeng}

\date{\today} 

\begin{document}
	\maketitle

	\begin{abstract}
		We consider the isoperimetric problem for the sum of two Gaussian densities in the line and the plane. We prove that the double Gaussian isoperimetric regions in the line are rays and that if the double Gaussian isoperimetric regions in the plane are half-spaces, then they must be bounded by vertical lines.
	\end{abstract}
	
	\maketitle
	\setcounter{tocdepth}{1}

	
	\section{Introduction}
	Sudakov-Tsirelson and Borell proved independently (see \cite[18.2]{morgan}) that for $\mathbb{R}^n$ endowed with a Gaussian measure, half-spaces bounded by hyperplanes are isoperimetric, i.e., minimize weighted perimeter for given weighted volume. Ca\~{n}ete et al. \cite[Question 6]{canete}, in response to a question of Brancolini, conjectured that for $\mathbb{R}^n$ endowed with a finite sum of Gaussian measures centered on the $x$-axis, half-spaces bounded by vertical hyperplanes are isoperimetric. We consider the case of two such Gaussians in $\mathbb{R}^1$ or $\mathbb{R}^2$. Our Theorem \ref{2.18} proves that on the double Gaussian line, rays are isoperimetric. Section 4 provides evidence that on the double Gaussian plane, half-spaces are isoperimetric.
	
	\begin{DGL}
		Theorem \ref{2.18} states that the isoperimetric regions in the double Gaussian line are rays. We may assume that the two Gaussians have centers at $1$ and $-1$. For small variances, the theorem follows by comparison with the single Gaussian. For larger variances, additional quantitative and stability arguments are needed to rule out certain non-ray cases.
	\end{DGL}
	
	\begin{DGP}
		A conjecture, stated in this paper as Conjecture \ref{4.1}, of Ca\~{n}ete et al. \cite[Question 6]{canete} states that isoperimetric regions in the double Gaussian plane are half-planes bounded by vertical lines. We use variational arguments to show that horizontal and vertical lines are the only lines that are candidates, and that vertical lines always beat horizontal lines.
	\end{DGP}
	
	\section{First and Second Variations}
	Formulas \ref{2.3} and \ref{2.5} state standard first and second variation formulas, analogous to the first and second derivative conditions for local minima of twice-differentiable real functions.
	\begin{definition}
		\label{2.31}
		A \textit{density} $e^\psi$ on $\mathbb{R}^n$ is a positive, continuous function used to weight volume and hypersurface area.
		Given a density $e^\psi$, the (weighted) \textit{volume} of a region $R$ is given by $$\int_R e^\psi\, dV_0.$$
		The (weighted) \textit{hypersurface area} of its boundary $\partial R$ is given by $$\int_{\partial R} e^\psi\, dA_0.$$
		$R$ is called \textit{isoperimetric} if no other region of the same weighted volume has a boundary with smaller hypersurface area.
	\end{definition}
	
	We now assume that the density $e^\psi$ is smooth. The existence and regularity of isoperimetric regions for densities of finite total volume is standard.
	
	\begin{exisreg}[{see \cite[5.5, 9.1, 8.5]{morgan}}]
		\label{2.32}
		Suppose that $e^\psi$ is a density in the line or plane such that the line or plane has finite measure $A_0$. Then for any $0 < A < A_0$, an isoperimetric region $R$ of weighted volume $A$ exists and is a finite union of  intervals bounded by finitely many points in the line or a finite union of regions with smooth boundaries in the plane.
	\end{exisreg}

	Let $e^\psi$ be a smooth density on $\mathbb{R}^{n+1}$. Let $R$ be a smooth region in $\mathbb{R}^{n+1}$. Let $\varphi_t$ be a smooth, one-parameter family of deformations on $\mathbb{R}^{n+1}$ such that $\varphi_0$ is the identity. For a given $x \in \partial R$, $\varphi_t(x)$ traces out a small path in $\mathbb{R}^{n+1}$ beginning at $x$ and $\varphi_t(\partial R)$ is a curve for each $t$.  Therefore $\{\varphi_t\}$ where $|t| < \epsilon$ describes a perturbation of $\partial R$. Define $$V(t) = \int_{\varphi_t(R)} e^\psi dV_0 ,\,\, P(t) = \int_{\varphi_t(\partial R)} e^\psi dA_0.$$ 
	
	\begin{firstvar}[{see \cite[Lemma 3.1]{rosal}}]
		\label{2.3}
		Suppose that $\mathbf{n}$ and $H$ are the inward unit normal and mean curvature of $\partial R$. Let $X$ be the vector field $d\phi_t/dt$ and $u = \langle X, \mathbf{n} \rangle$. Then we have that
		$$V'(0) = -\int_{\partial R} e^\psi u \, dA_0, \,\, P'(0) = -\int_{\partial R} (nH - \langle \nabla{\psi}, \mathbf{n} \rangle) e^\psi u \, dA_0.$$ 
	\end{firstvar}

	Since any isoperimetric curve is a local minimum among all curves enclosing a certain volume $A$, it satisfies $P'(0) = 0$ for any $\varphi_t$ such that $V(t) = A$ for small $t.$
	
	\begin{corollary}
		\label{2.4}
		If a curve $\partial R$ is isoperimetric, then $(nH - \langle \nabla{\psi}, \mathbf{n} \rangle)$ is constant on $\partial R$.
	\end{corollary}

	\begin{proof}
		If a curve $\partial R$ is isoperimetric, then it satisfies $P'(0) = 0$. By Formula \ref{2.3}, this occurs if and only if  $(nH - \langle \nabla{\psi}, \mathbf{n} \rangle)$ is constant on $\partial R$.
	\end{proof}

	\begin{definition}
		\label{4.4}
		Let $C$ be a boundary in the line or plane with unit inward normal $\textbf{n}$ and let $\kappa$ denote the standard curvature. For a density $e^\psi$, we call $\kappa_\psi = \kappa - d\psi/d\textbf{n}$ the \textit{generalized curvature} of $C$. 
	\end{definition}
	
	By Corollary \ref{2.4}, all isoperimetric curves have constant generalized curvature. In the real line, $n=0$, so that isoperimetric curves have $\langle \nabla \psi, \mathbf{n} \rangle$ constant. For the interval $[a,b]$, the generalized curvature evaluated at $b$ is equal to $\psi ' (b)$ while the generalized curvature evaluated at $a$ is equal to $-\psi ' (a)$. 
	
	\begin{secondV_0ar}[{see \cite[Proposition 3.6]{rosal}}]
		\label{2.5}
		Let the real line be with smooth density $e^\psi$. If a one-dimensional boundary $l = \partial R$ satisfies $P'(0) = 0$ for any volume-preserving $\{\varphi_t\}$, then $$(P - \kappa_\psi V)''(0) = \int_l fu^2(\frac{d^2\psi}{dx^2}) \,da.$$
	\end{secondV_0ar}
	
	\begin{proof}
		This formula comes from Proposition 3.6 in \cite{rosal}, where the second variation is stated for arbitrary dimensions. Some terms from the general formula cancel in the one-dimensional case. 
	\end{proof}
	
	\begin{corollary}
		\label{2.7}
		Let $S$ be a subset of the real line such that $\psi '' (x) \leq 0$ for all $x \in S$ with equality holding at no more than one point. If $B$ is an isoperimetric boundary contained in $S$, then $B$ is connected and thus a single point.
	\end{corollary}
	
	\begin{proof}
		If $B$ has at least two connected components, then since by Proposition \ref{2.32} $B$ consists of a finite union of points, there is a nontrivial volume-preserving flow on $B$ given by moving one component so as to increase the volume and the other so as to decrease it. By the Second Variation Formula \ref{2.5}, the second variation satisfies $$(P - \kappa_\psi V)''(0) = \int_B fu^2(\psi '' (x)) \,da < 0.$$ This contradicts that $B$ is isoperimetric. 
	\end{proof}
	
	
	
	\section{Isoperimetric Regions on the Double Gaussian Line}
	
	Theorem \ref{2.18} states that for the real line with density given by the sum of two Gaussians with the same variance $a^2$, isoperimetric regions are rays bounded by single points. This theorem is a necessary condition for Conjecture \ref{4.1}, which states that isoperimetric regions in the double Gaussian plane are half-planes bounded by vertical lines. Propositions \ref{2.8}, \ref{2.22}, and \ref{2.10} treat the cases $a^2 \geq 1$, $1 > a^2 > 1/2$, and $1/2 \geq a^2 > 0$. 
	
	Lemma \ref{2.9} shows that the Gaussians having the same variance allows us to reduce the problem to ruling out a few non-interval but still symmetrical cases. When the Gaussians have different variances, the problem is harder and not treated by our results.
	
	Let $g_{c,a}$ denote the Gaussian density with mean $c$ and variance $a^2$, and let $$f_{c,a}(x) = \dfrac{1}{2}(\dfrac{e^{-(x-c)^2/2a^2} + e^{-(x+c)^2/2a^2}}{a\sqrt{2\pi}}) = \dfrac{1}{2}(g_{c,a}(x) + g_{-c,a}(x)).$$ Let $$f(x) = \frac{1}{2}(f_1(x) + f_2(x)) = \frac{1}{2}(g_{1,a}(x) + g_{-1,a}(x)).$$  In one dimension, the regions are unions of intervals and their boundaries are points. Since the total measure is finite, isoperimetric regions exist by Proposition \ref{2.32}.  For a given weighted length $A$, we seek to find the set of points with the smallest total density which bounds a region of weighted length $A$. Since the complement of a region of weighted length $A$ has weighted length $1-A$, we can assume that our regions have weighted length $0 \leq A \leq 1/2$.
	
	The following proposition shows that it suffices to consider the density $f$.
	
	\begin{proposition}
		\label{2.1}
		Suppose that $B$ is an isoperimetric boundary enclosing $a$ region $L$ of weighted length $A$ for the density $f_{1,a}(x)$. Then for any $b > 0$, $bB$ is an isoperimetric boundary enclosing region $bL$ of weighted length  $A$ for the density $f_{b,ab}(x).$
	\end{proposition}
	
	\begin{proof}
		Let $g$ denote the standard Gaussian density.
		
		First, we show that for any boundary $P$ enclosing a region $Q$, the weighted length of $bQ$ for the density $f_{b,ab}(x)$ is the same as the weighted length of $Q$ for the density $f_{1,a}(x)$. We have that $$|Q| = \int_Q f_{1,a}(x) dx = \frac{1}{2}\int_Q g_{1,a}(x) dx + \frac{1}{2}\int_Q g_{-1,a}(x) dx$$
		$$= \frac{1}{2}\int_{(Q-1)/a} g(x) dx + \frac{1}{2}\int_{(Q+1)/a} g(x) dx = \frac{1}{2}\int_{(bQ-b)/(ab)} g(x) dx + \frac{1}{2}\int_{(bQ+b)/(ab)} g(x) dx $$
		$$=\frac{1}{2}\int_Q g_{b,ab}(x) dx + \frac{1}{2}\int_Q g_{-b,ab}(x) dx = |bQ|,$$ where the $|...|$ denotes the weighted length in the appropriate densities.
		
		Second, for any two boundaries $P_1$ and $P_2$, we have that $f_{b,ab}(bx) = \frac{1}{b}f_{1,a}(x)$ for $x \in P_i.$ Thus, $|P_1| \geq |P_2|$ in the density $f_{1,a}(x)$ exactly when $|bP_1| \geq |bP_2|$ in the density $f_{b,ab}(x)$. 
		
		Therefore $|bL| = A$ in the density $f_{b,ab}(x)$, and if any other boundary $P$ enclosing region $Q$ satisfies $|Q| = A$ in the density $f_{b,ab}(x)$, then since $B$ is isoperimetric, we have that $|B| \leq |P/b|$ in the density $f_{1,a}(x)$. Therefore $|bB| \leq |P|$ in the density $f_{b,ab}(x)$, so that $bP$ is isoperimetric.
	\end{proof}
	
	As a result of Proposition \ref{2.1}, it suffices to consider the density $$f = \frac{1}{2}(f_1 + f_2) = \frac{1}{2}(g_{1,a} + g_{-1,a}).$$.
	
	\begin{figure}[h]
		\centering
		\begin{subfigure}{0.45\textwidth}
			\centering
			\includegraphics[width=\linewidth]{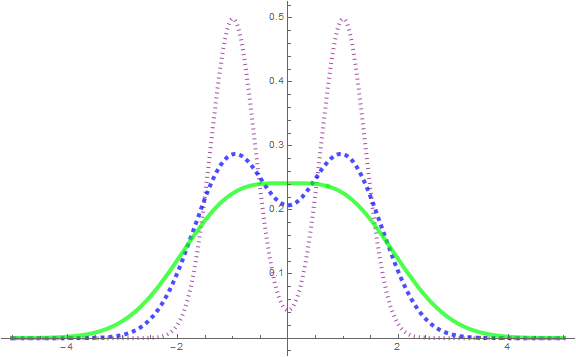}
			\caption{ $f(x)$}
		\end{subfigure}%
		\begin{subfigure}{0.45\textwidth}
			\centering
			\includegraphics[width=\linewidth]{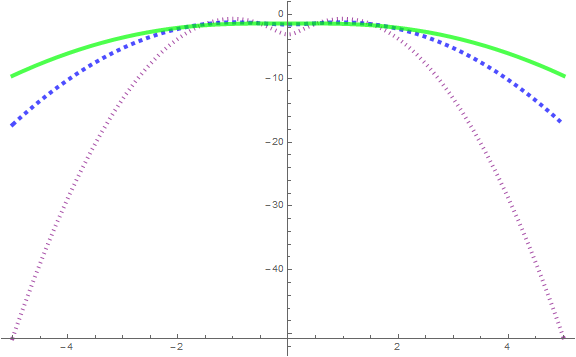}
			\caption{$\psi(x)$}
		\end{subfigure}
		\caption{Plots of $f$ and $\psi$. The purple curves are for $a^2 = 0.16,$ the blue curves for $a^2 = 1/2$, and the green curves for $a^2 = 1$.}
		\label{fg2.31}
	\end{figure}

	\begin{proposition}
		\label{2.2}
		Let $X$ be the disjoint union of two real-lines $X_1$ and $X_2$, each with a standard Gaussian density scaled so that it has weighted length $1/2$. For any given length $0 < A < 1/2$, the isoperimetric region in $X$ of length $A$ is a ray contained entirely in $X_1$ or $X_2.$
	\end{proposition}
	
	\begin{proof}
		Let $B$ be an isoperimetric boundary and $B_i$ its intersection with $X_i$. If either $B_1$ or $B_2$ are nonempty, then they each must be single points since the isoperimetric boundaries for the single Gaussian are always single points. Assume, in contradiction to the proposition, that for $i=1,2$, $B_i = \{b_i\}$ is the $i$-th component on the $i$-th Gaussian bounding a ray $L_i$ of weighted length $A_i$. Since $A_1 + A_2 < 1/2$, it is possible to put a point $b_1'$ on the 1st Gaussian at the same height as that of $b_2$ bounding a ray $L_1'$ disjoint from $L_1$ and with weighted length $A_2$. Consider the boundary $B' = \{b_1,b_1'\}$, which has the same weighted perimeter as that of $B$. There exists a single point on $B_1$ bounding a ray of area $A$ and with weighted density smaller than $|B'| = |B|$. This contradicts the fact that $B$ is isoperimetric.
	\end{proof}

	\begin{proposition}
		\label{2.33}
		For the double Gaussian density $f$, the log derivative $\psi'$ is given by $$\psi ' (x) = a^{-2}(-x + \tanh\frac{x}{a^2}).$$
	\end{proposition}
	
	\begin{proof}
		We have that
		$$\psi'(x) = \dfrac{\dfrac{-e^{-\frac{(-1+x)^2}{2a^2}}(-1+x)}{a^2} + \dfrac{-e^{-\frac{(1+x)^2}{2a^2}}(1+x)}{a^2}}{e^{-\frac{(-1+x)^2}{2a^2}} + e^{-\frac{(1+x)^2}{2a^2}}}.$$
		
		By using the substitution 
		$$\tanh(x/a^2) = (e^{x/a^2} - e^{-x/a^2})/(e^{x/a^2} + e^{-x/a^2}),$$ we get that
		$$\psi ' (x) = a^{-2}(-x + \tanh\frac{x}{a^2}).$$
	\end{proof}
	
	\begin{proposition}
		\label{2.8}
		For the double Gaussian density $f$, if $a \geq 1$, isoperimetric boundaries are single points.
	\end{proposition}
	
	\begin{proof}
		For any given $a$, we have that  $$\psi ' (x) = a^{-2}(-x + \tanh\frac{x}{a^2}),$$$$ \psi '' (x) = a^{-4}(-a^2+\sech^2\frac{x}{a^2}),$$
		and
		$$\psi ''' (x) = -2a^{-6}\sech^2\frac{x}{a^2}\tanh\frac{x}{a^2}.$$  
		\begin{figure}[h]
			\centering
			\begin{subfigure}{0.45\textwidth}
				\centering
				\includegraphics[width=\linewidth]{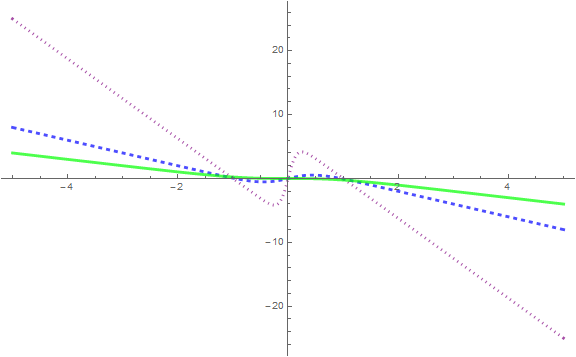}
				\caption{ $\psi'(x)$}
			\end{subfigure}%
			\begin{subfigure}{0.45\textwidth}
				\centering
				\includegraphics[width=\linewidth]{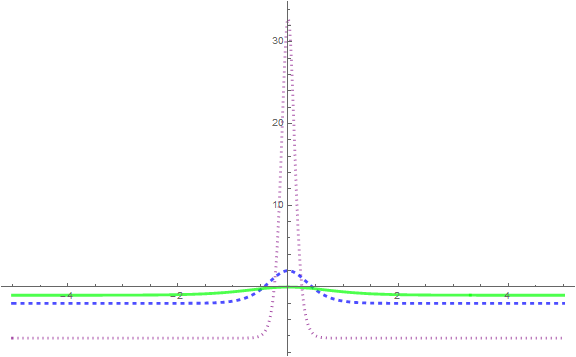}
				\caption{$\psi''(x)$}
			\end{subfigure}
			\begin{subfigure}{0.45\textwidth}
				\centering
				\includegraphics[width=\linewidth]{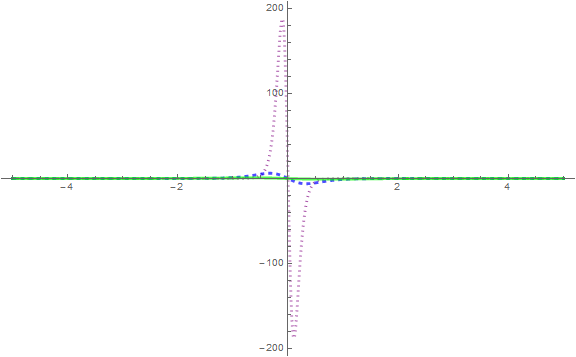}
				\caption{$\psi'''(x)$}
			\end{subfigure}
			\caption{Plots of $\psi',\psi'',$ and $\psi'''$. The purple curves are for $a^2 = 0.16,$ the blue curves for $a^2 = 1/2$, and the green curves for $a^2 = 1$.}
			\label{fg2.32}
		\end{figure}
		
		As shown in Figure \ref{fg2.32}, $\psi ''' (x)$ is positive for any $x< 0$ and negative for $x > 0$, so that $\psi '' (x)$ achieves its unique maximum at $x = 0$ for any given $a$. We have that $\psi ''(0) = (1-a^2)/(a^4)$, so that $\psi''(0)$ is greater than $0$ for $a < 1$, and less than or equal to $0$ for $a \geq 1$. If $a \geq1$, by  Corollary \ref{2.7}, isoperimetric boundaries are always connected. Since isoperimetric boundaries consist of finite unions of points, they must be single points.
	\end{proof}
	
	\begin{lemma}
		\label{2.9}
		Let $p$ and $q$ be two real functions with $p(0) = q(0)$. Suppose $p$ and $q$ satisfy
		\begin{enumerate}
			\item $p'(0) = q'(0) \geq 0$,
			\item $q''(0) \geq p''(0)$ ,
			\item $q''(0) \geq 0$, and
			\item $p''' < 0$ and $q''' > 0$ on $(0, \infty)$.
		\end{enumerate}
		For any $a, b> 0$, if $p(a) = q(b)$, then $q'(b) > p'(a)$.
	\end{lemma}
	\begin{figure}[h]
		\includegraphics[height=3in]{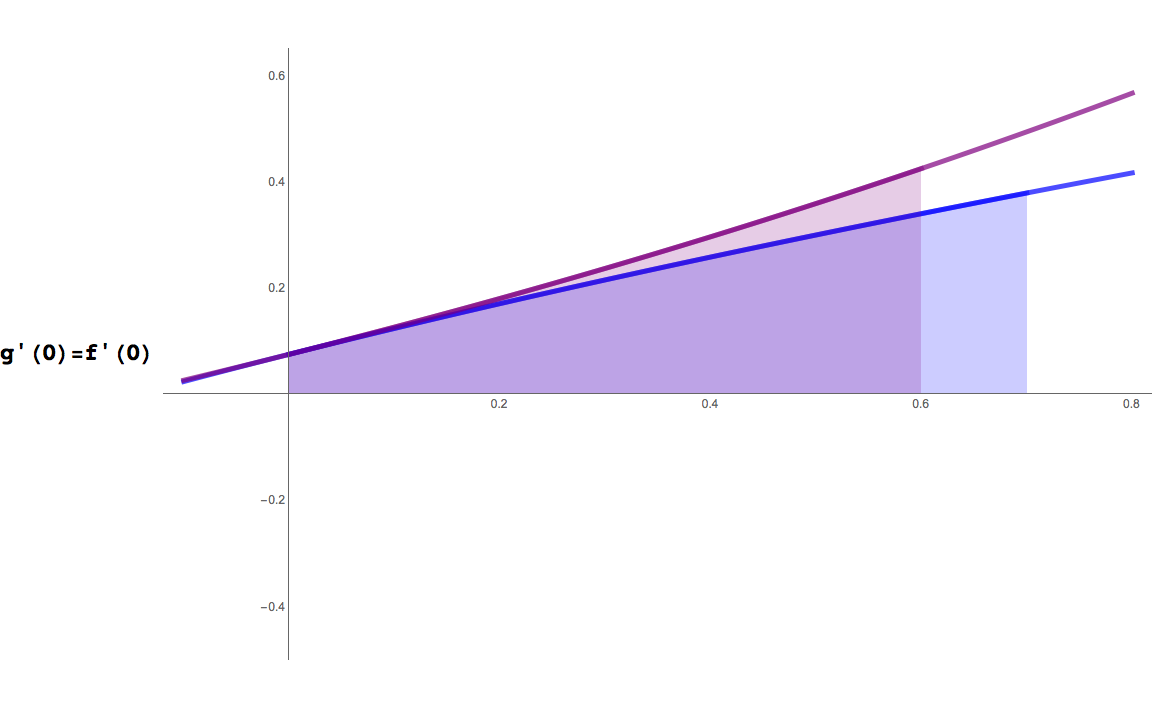}
		\caption{$q'$ is blue while $p'$ is purple. When the areas are equal as in the picture, $q'$ is higher. }
		\label{fg2.3}
	\end{figure}
	
	\begin{proof}
		As in Figure \ref{fg2.3}, for all $x > 0,$ by (2) and (4) $q''(x) > p''(x)$ and by (3) and (4) $q''(x) > 0$. If we choose $a'$ so that $q'(a') = p'(a)$, we will have that $a' < a$.  
		\newpage
		Since by (4) $p'$ is concave and $q'$ is convex,
		\[
		q(a') = \int_0^{a'} q'(t)\,dt \leq \frac{1}{2}(a'*q'(a'))\]
		\[< \frac{1}{2}(a*p'(a)) \leq \int_0^{a'} q'(t)\,dt = p(a) = q(b).\]
		
		Therefore $b > a'$, so that $q'(b) > p'(a)$, as asserted.
	\end{proof}
	
	\begin{proposition}
		\label{2.12}
		Suppose that $[s,t]$ is an interval of $f$-weighted length $0 < A < 1/2$ with $-1 < s < t < 1$. Then there exists a union of rays $B = (-\infty, c] \cup [d,\infty]$ of $f_1$-weighted length $A$ such that $f_1(c) <f_1(t) < f(t)$ and $f_1(d) < f_2(s) < f(s).$
	\end{proposition}
	
	\begin{figure}[h]
		\centering
		\begin{subfigure}{0.5\textwidth}
			\centering
			\includegraphics[width=\linewidth]{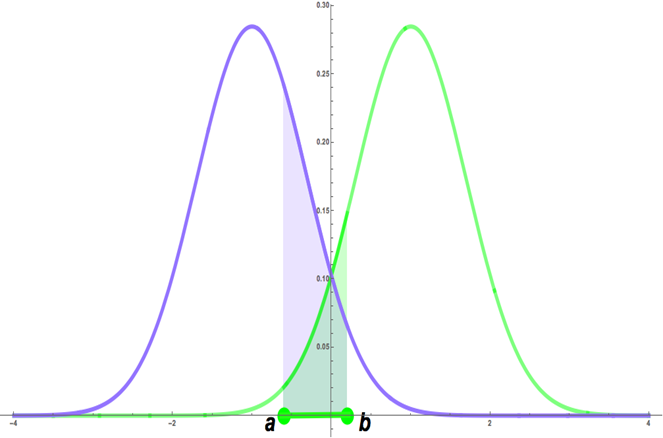}
			\caption{Interval in the double Gaussian}
		\end{subfigure}%
		\begin{subfigure}{0.5\textwidth}
			\centering
			\includegraphics[width=\linewidth]{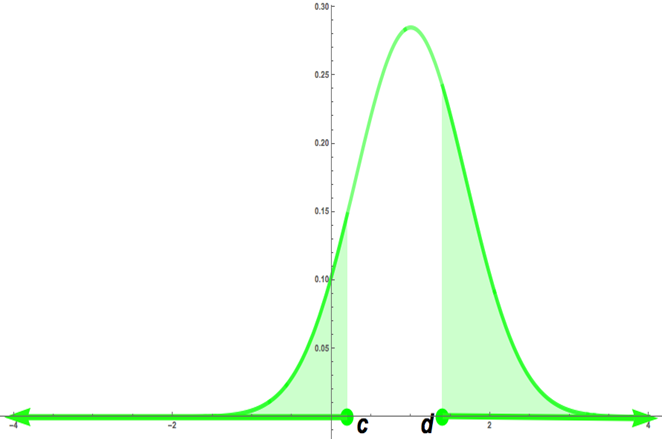}
			\caption{Two rays in the single Gaussian}
		\end{subfigure}
		\caption{The total areas are the same, but the heights in (B) are slightly lower.}
		\label{fg2.7}
	\end{figure}
	\begin{proof}
		Since $2+s = 1 + (s-(-1))$, we have that $f_1(2+s) = f_2(s)$. The union of rays $(-\infty,t] \cup [2+s,\infty)$ has greater $f_1$-weighted length than the $f$-weighted length of $[s,t]$. Therefore there exists $c < t$ and $d > 2+s$ such that $(-\infty,c] \cup [d,\infty)$ has $f_1$-weighted length $A$, and $$f_1(c) + f_1(d) < f_1(t) + f_2(s) < f(t) + f(s).$$
	\end{proof}

	\begin{proposition}
		\label{2.13}
		If $[s,\infty)$ has $(1/2)f_1$-weighted length $0 < A \leq 1/4$, then there exists $t > s$ such that $[t, \infty)$ has $f$-weighted length $A$.  
	\end{proposition}
	\begin{figure}[h]
		\centering
		\begin{subfigure}{0.5\textwidth}
			\centering
			\includegraphics[width=\linewidth]{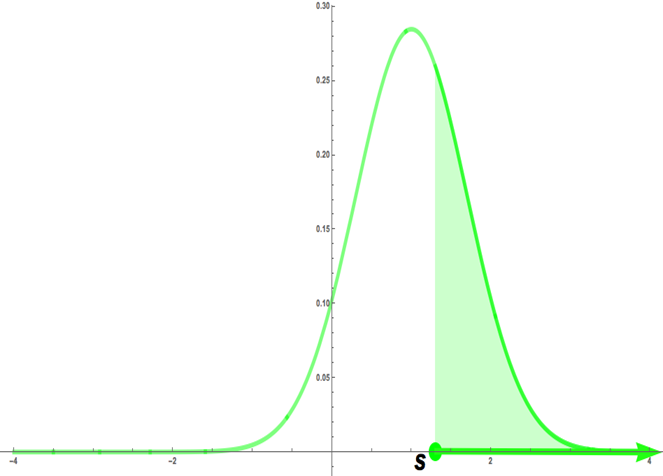}
			\caption{Ray in the single Gaussian}
			\label{fig:sub1}
		\end{subfigure}%
		\begin{subfigure}{0.5\textwidth}
			\centering
			\includegraphics[width=\linewidth]{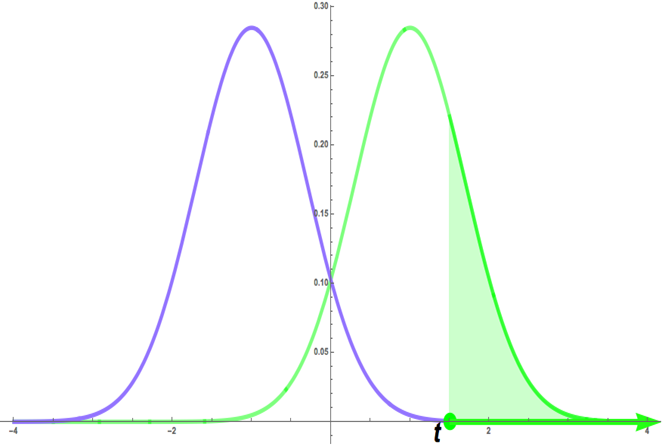}
			\caption{Ray in the double Gaussian}
			\label{fig:sub2}
		\end{subfigure}
		
		\caption{The total areas are the same, but the height in (B) is slightly lower.}
		\label{fg2.9}
	\end{figure}
	
	\begin{proof}
		If $[s,\infty)$ has $(1/2)f_1$-weighted length $0 < A \leq 1/4$, then $s \geq 1.$ The interval $[s, \infty)$ has $f$-weighted length greater than $A$. Therefore there exists $t > s$ such that $[t, \infty)$ has $f$-weighted length $A$.
	\end{proof}
	
	Now we begin analyzing the case where the variance satisfies $0 < a^2 < 1$.
	
	\begin{proposition}
		\label{3.36}
		If $a^2$ satisfies $0 < a^2 \leq 1$, then $\psi''(x) = 0$ exactly when $x = \pm a^2\arccosh (1/a)$.
	\end{proposition}
	\begin{proof}
		This follows from the formula for $\psi''(x)$ given in Proposition \ref{2.8}.
	\end{proof}
	
	Suppose that $a^2$ is a variance. In the proof of the following proposition, we will use the quantity
	\[
	c_a = a^2\arccosh (1/a).
	\]
	\begin{proposition}
		\label{2.14}
		Suppose that $0 < a^2 \leq 1$ and $B$ is an isoperimetric boundary with at least one point $s$ in $[0,c]$ where $c = c_a,$ enclosing a region of weighted-length $0 < A < 1/2$. Then the boundary $B$ is one of the following:
		\begin{enumerate}
			\item a single point $s$ enclosing the ray $[s, \infty)$,
			\item \{s, t\} where $t > s$ enclosing the interval $[s,t]$,
			\item \{s, t\} where $s >  0 > t$ enclosing the interval $[t,s]$,
			\item \{s, -s, t\} enclosing $[-s,s] \cup (-\infty, t]$, $[-s,s] \cup [t,\infty)$ or $[s,t] \cup (-\infty, -s]$.  
		\end{enumerate}
		The analogous claims apply if $s  \in [-c,0]$.
	\end{proposition}
	
	\begin{figure}[h]
		\includegraphics[height=2in]{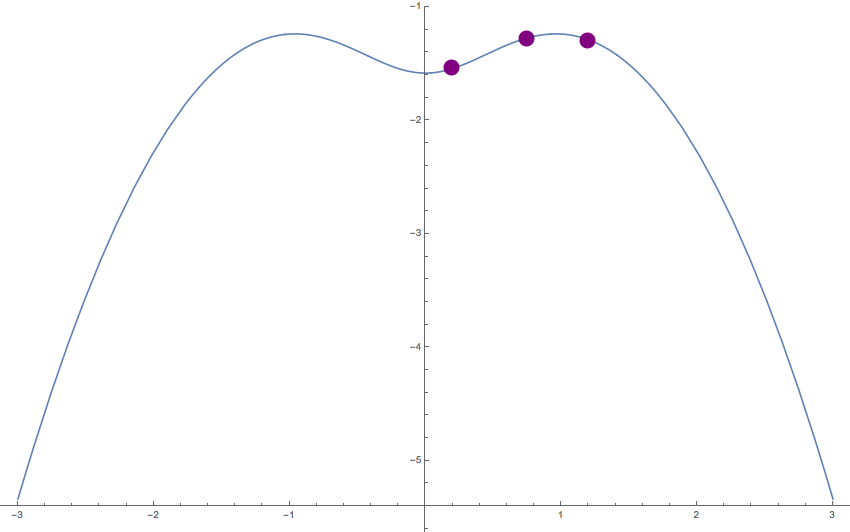}
		\caption{On the graph of $\psi = \log f$, there are at most three points with $x > 0$ with the same value for $|\psi'(x)|$.}
		\label{fg2.10}
	\end{figure}
	
	\begin{proof}
		Since $B$ is isoperimetric, it can contain at most one point $x$ at which $\psi''(x) < 0.$ If it contained two such points, then by slightly shifting the two points we could create a new region with the same weighted length. By the second variation formula, the boundary of this region would have a smaller total density. Therefore $B$ can contain at most one point outside of $[-c,c]$. 
		
		In addition, $B$ has constant curvature, so that $|\psi'|$ is constant on $B$ (see Figure \ref{fg2.10}). Since $\psi''(s)$ is positive on $[0, c)$ and negative on $(c, \infty]$, there exists one point $t > s > 0$ such that $\psi'(t) = \psi'(s)$ and one point $u > t > s > 0$ such that $-\psi'(u) = \psi'(s).$ Therefore $B$ is a subset of $\{s, t, u, -s, -t , -u\}.$  Suppose $B$ is not $(1)$. If $B$ contains no points outside of $[-c,c]$, then $B$ is $(3)$. Suppose $B$ contains one point $y$ outside of $[-c,c]$. If $t > 0$, then the only possibilities are $(2)$ or $(4)$. If $t < 0$, then the only possibilities are $(3)$ or $(4)$. The regions enclosed follow from the fact that we assume $0 < A < 1/2$.
	\end{proof}
	
	\begin{proposition}
		\label{2.16}
		Suppose that $B$ is an isoperimetric boundary with at least one point $s \in [-c,c].$ If $B$ is of type \ref{2.14}\emph{(3)} and  $0 < a^2 \leq 1/2$, then the region $R$ enclosed by $B$ has $f$-weighted length no more than $1/4$.
	\end{proposition}
	
	\begin{proof}
		
		We have that $$\dfrac{d}{dx}(x - \arccosh(x)) = 1 - \dfrac{1}{\sqrt{x-1}\sqrt{1+x}} > 0$$ for $x > \sqrt{2},$ so that $x - \arccosh(x)$ is increasing on $(\sqrt{2}, \infty)$. 
		If $y = 1/x,$ then the function $y - \arccosh(y) - 0.5$ decreases on $(0, 1/\sqrt{2})$. Since $\sqrt{2} - \arccosh(\sqrt{2}) - 0.5 > 0$, we have that $\arccosh(y) < y - 0.5$ on $(0, 1/\sqrt{2})$. Therefore $$c < a - 0.5a^2 \leq 1/\sqrt{2} - 1/4 < 1/2.$$
		
		Consider the function $$I(x) = \int_{x-1}^x f_1(x) dx + \int_{x-1}^x f_2(x) dx$$ which sends $x$ to the weighted length of $[x-1,x].$  Then $$I'(x) = f_1(x) - f_1(x-1) + f_2(x) - f_2(x-1) = [f_2(x) - f_1(x-1)] + [f_1(x) - f_2(x-1)].$$ For $|x| < 1/2$, both the bracketed quantities are negative, so that $I$ is decreasing on $[0,c]$. We have that $$I(0) = \int_{-1}^0 f_2(x) dx + \int_{-1}^x f_1(x) dx= \int_{-1}^1 f_2(x) dx < \int_{-1}^\infty f_2(x) dx = 1/4.$$ 
		
		Therefore if we can show that $s - t \leq 1$, we will have that the $f$-weighted length of $[s,t]$ is less than $I(s) \leq 1/4$ and be done. This follows immediately when $t = -s$, since $s \leq c <  1/2.$ When $t \neq -s$, we observe that $s-1$ is to the left of $-c$, so it suffices to show that $\psi'(s-1) \geq \psi'(t) = -\psi'(s)$. Thus we want to show that $$\psi'(s-1) + \psi'(s) = \psi'(s) - \psi'(1-s)$$$$  = ([(1 - s) -s] + [\tanh(s/a^2) - \tanh((1-s)/a^2)])/(a^2) \geq 0$$ on $[0,1/2]$.
		
		This is equivalent to showing that $$\gamma(s) := ([(1 - s) -s] + [\tanh(s/a^2) - \tanh((1-s)/a^2)]) \geq 0$$ on $[0,c]$. Since $|\tanh|< 1$, $\gamma(0) > 0$. In addition, $\gamma(1/2) = 0.$ Therefore it suffices to show that $\gamma$ achieves its minimum value on $[0,1/2]$ at $s = 1/2.$ We will do this by using the first derivative test to show that there is only one other local extremum in the interval and further demonstrating that this local extremum is not the minimum point.
		
		We have that $$\gamma'(s) = \sech^2(s/a^2)/a^2 + \sech^2((1-s)/a^2)/a^2 - 2.$$  Since $1/a^2 \geq 2$, we have that $\gamma'(0) > 0$. In addition, $$\gamma'(1/2) =2\sech^2(1/(2a^2))/a^2 - 2.$$
		
		By using the substitution $$\sech^2(x) = 4/(e^{2x} + e^{-2x} +2),$$ we get that $$\sech^2(1/2x) = 4/(e^{1/x} + e^{-1/x} + 2) \leq 4/(e^{1/x} + 2).$$ Therefore $$\sech^2(1/2x)(1/x) \leq 4/(xe^{1/x} + 2x).$$ We have that $$\alpha(x) := (xe^{1/x}+2x)' = (2 + e^{1/x} - e^{1/x}/x).$$ When $0 < x \leq 1/2$, we have that $$\alpha(x) \leq 2 + e^{1/x} - 2e^{1/x} = 2 - e^{1/x} \leq 2 - e^2 < 0.$$ Therefore $\alpha(x)$ attains a minimum value of $e^2/2 + 1 > 4$ on $(0, 1/2]$. This shows that $$\sech^2(1/2x)(1/x) \leq 4/(xe^{1/x} + 2x) < 1$$ on $(0,1/2]$, so that $\gamma'(1/2) < 0.$
		
		By the intermediate value theorem, there exists $z_1 \in (0, 1/2)$ such that $\gamma'(z_1) = 0.$ It follows that $z_2 = 1 - z_1 > 1/2$ is also a zero of $\gamma'$. Now $\sech^2(x) = \sech^2(-x)$ tends to $0$ as $x$ tends to $\infty$, so that $\gamma' < 0$ for some $s << 0$. Therefore there exists $z_3$ in $(-\infty , 0)$ such that $\gamma'(z_3) = 0$, and $z_4 = 1 - z_3 > 1$ is also a zero of $\gamma'$.
		
		Again using the substitution  $$\sech^2(x) = 4/(e^{2x} + e^{-2x} +2),$$ we see that $\gamma'(s)$ is a rational function of $e^{2s/a^2}$ whose numerator is quartic. Therefore $\gamma'$ has at most $4$ zeros, so that $z_1$ is the only zero of $\gamma'$ in $(0, 1/2).$ Since $\gamma'(0) > 0$, $$\gamma(z_1) > \gamma(0) > \gamma(1/2),$$ so that $\gamma(s) \geq \gamma(1/2) = 0$ for $s \in [0,1/2].$
	\end{proof}
	
	\begin{figure}[h]
		\centering
		\includegraphics[height=2in]{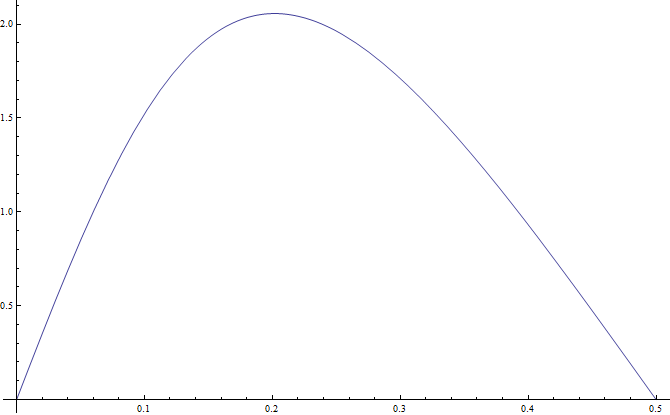}
		\caption{$\psi'(s) - \psi'(1-s)$}
		\label{ev}
	\end{figure}
	
	\begin{proposition}
		\label{2.17}
		If the variance $0 < a^2 \leq 1/2$, then the isoperimetric boundaries $B$ with one point $b$ in $[0,c]$ cannot be of type \ref{2.14}$(3)$.
	\end{proposition}
	\begin{figure}[h]
		\centering
		\begin{subfigure}{0.45\textwidth}
			\centering
			\includegraphics[width=\linewidth]{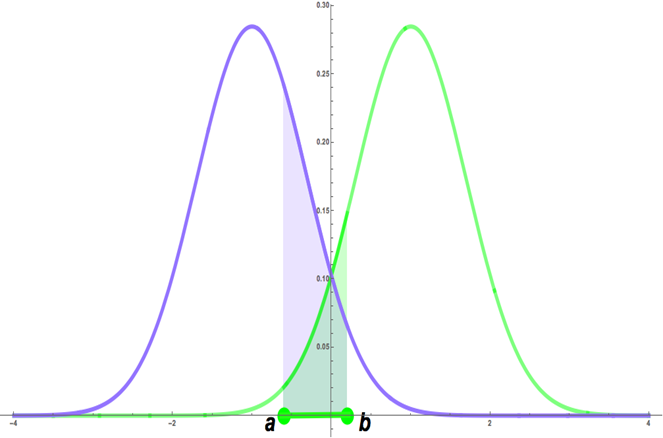}
			\caption{Interval in the double Gaussian}
			\label{fig:sub1}
		\end{subfigure}%
		\begin{subfigure}{0.45\textwidth}
			\centering
			\includegraphics[width=\linewidth]{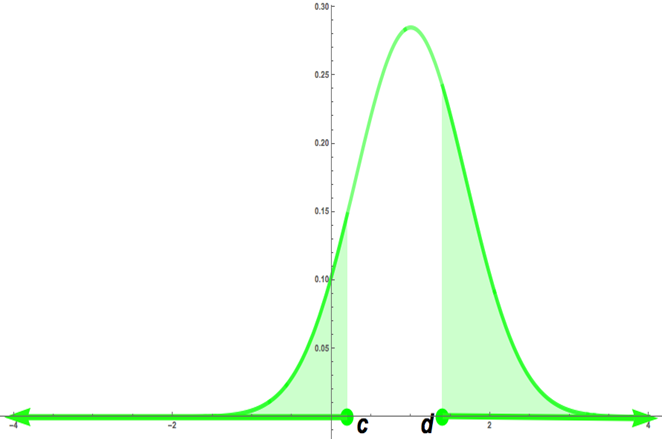}
			\caption{Two rays in the single Gaussian}
			\label{fig:sub2}
		\end{subfigure}\\
		\centering
		\begin{subfigure}{0.45\textwidth}
			\centering
			\includegraphics[width=\linewidth]{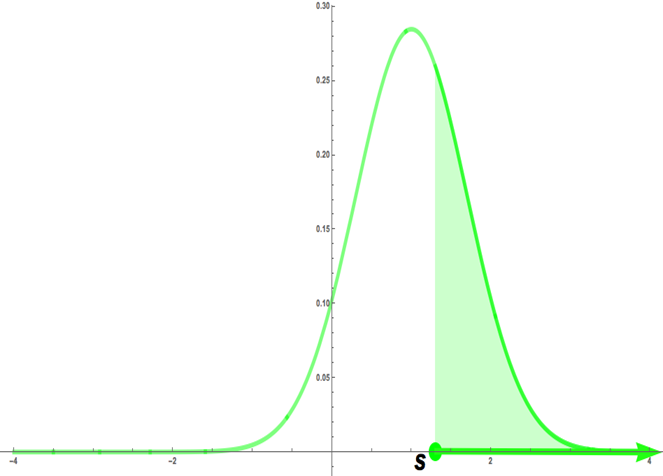}
			\caption{Ray in the single Gaussian}
			\label{fig:sub1}
		\end{subfigure}%
		\begin{subfigure}{0.45\textwidth}
			\centering
			\includegraphics[width=\linewidth]{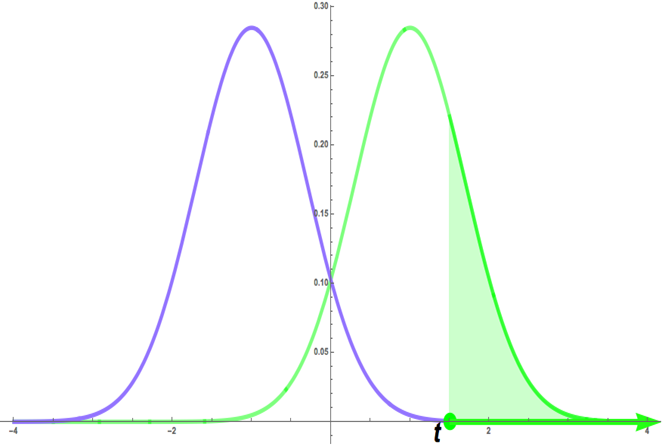}
			\caption{Ray in the double Gaussian}
			\label{fig:sub2}
		\end{subfigure}
		\caption{When all the areas are the same, we have $(A) > (B) > (C)$ and $(D) > (C)$}
		\label{fg2.11}
	\end{figure}
	
	\begin{proof}
		Let $A$ be the weighted length of $B$. If, in contradiction to the proposition, $B$ is of type \ref{2.14}$(3)$, then $B$ is of the form $[a,b]$ where $-1 < a < b < 1$, as shown in Figure \ref{fg2.11}A. By Proposition \ref{2.12}, there exists a union of rays $(-\infty,c] \cup [d,\infty)$ with $f_1$-weighted length $A$ such that $f_1(c) + f_1(d) < f(a) + f(b)$. This is shown in Figure \ref{fg2.11}B. By the solution  to the single Gaussian isoperimetric problem, there exists a ray $[s, \infty)$, as shown in Figure \ref{fg2.11}C, with $f_1$-weighted length $A$ such that $f_1(t) < f_1(c) + f_1(d)$. By Proposition \ref{2.16}, $A \leq 1/4$, so that $s \geq 1.$ By Proposition \ref{2.13}, there exists a ray $[t, \infty)$, as shown in Figure \ref{fg2.11}D, with $f$-weighted length $A$ such that $t > s.$
		
		To get a contradiction to the fact that $B$ is isoperimetric, we show that $(f(a) + f(b)) - f(t) > 0.$ Write $$(f(a) + f(b)) - f(t)= [(f(a) + f(b)) - (f_1(c) + f_1(d))]$$$$ + [(f_1(c) + f_1(d)) - f_1(s)] + [f_1(s) - f(t)].$$ Since $[(f_1(c) + f_1(d)) - f_1(s)] > 0$, it suffices to show that $[(f(a) + f(b)) - (f_1(c) + f_1(d))] > [f(t) - f_1(s)]$. Since $f(a) > f_1(d)$, we have that $$[(f(a) + f(b)) - (f_1(c) + f_1(d))] > f(b) - f_1(c) > f(b) - f_1(b) = f_2(b).$$ Since $f(t) < f(s)$, we have that $$[f(t) - f_1(s)] < f(s) - f_1(s) = f_2(s).$$ Since $-1 < b < 1 < s$, we have that $f_2(s) < f_2(b)$, and this proves the claim. 
	\end{proof}
	
	\begin{proposition}
		\label{2.21}
		If the variance $a^2 \leq 1/2$, then the isoperimetric boundaries $B$ with one point $b >0 $ in $[-c,c]$ cannot be of type \ref{2.14}\emph{(2)}.
	\end{proposition}
	
	\begin{figure}[h]
		\centering
		\begin{subfigure}{0.5\textwidth}
			\centering
			\includegraphics[width=\linewidth]{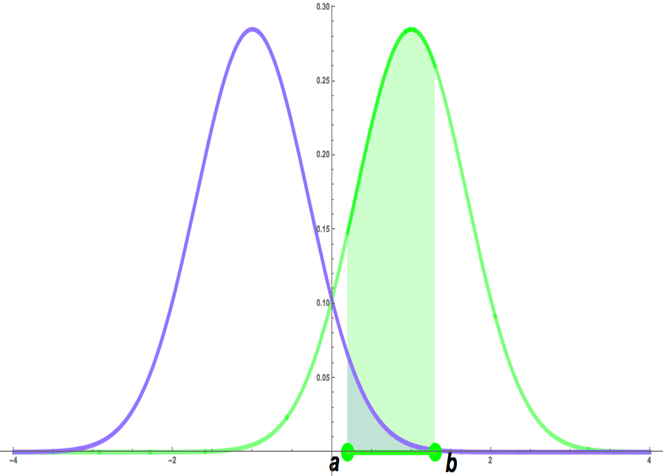}
			\caption{Original ray}
			\label{fig:sub1}
		\end{subfigure}%
		\begin{subfigure}{0.5\textwidth}
			\centering
			\includegraphics[width=\linewidth]{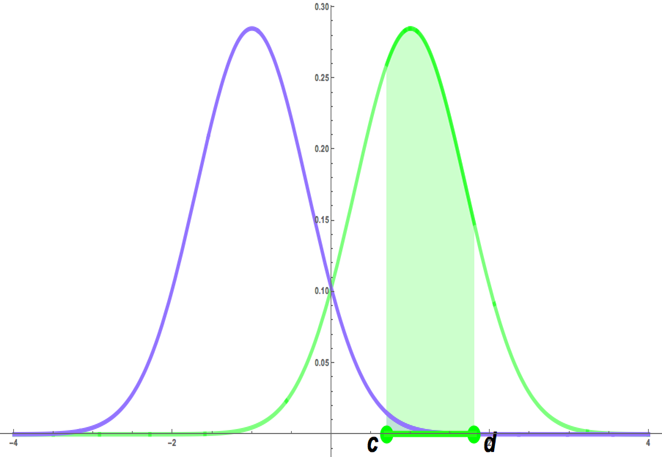}
			\caption{Reflected ray}
			\label{fig:sub2}
		\end{subfigure}
		\caption{When all the areas are the same, we have $(A) > (B) > (C)$ and $(D) > (C)$}
		\label{fg2.21}
	\end{figure}
	
	\begin{proof}
		We know that $f(a) < f(b)$ [recall the concavity/convexity argument], and since $f_2(b) < f_2(a)$, we must have that $f_1(a) < f_1(b).$
		
		Pick $d > c$ such that $f_1(c) = f_1(b)$ and $f_1(d) = f_1(a)$ . In other words, we get $[c,d]$ by reflecting $[a,b]$ over the line $x = 1$. Since $a < 1$, either have $c < 1 < d$ or $1 < d < c$.
		
		In the first case, we have that $[c,d]$ has the same $f_1$-length as $[a,b]$, and since $c > a$ and $d > b$, we have that $f_2(c) < f_2(a)$ and $f_2(d) < f_2(b)$. Therefore $f(c) + f(d) < f(a) + f(b)$. At the same time, the $f_2$-length of $[c,d]$ is less than that of $[a,b]$. This difference is at most the $f_2$ length of $[a,\infty)$. Since $f_1(d) = f_1(a) > f_2(a)$, we can find $e > d$ such that $[c,e]$ has $f$-length $A$. In addition, $f(c) + f(e) < f(c) + f(d) < f(a) + f(b)$, so that $[a,b]$ is not isoperimetric.
		
		In the second case, we have that $[d,c]$ has the same $f_1$-length as $[a,b]$, and since $d,c > a,b$, we have that $f_2(d) < f_2(a)$ and $f_2(c) < f_2(b)$. Therefore $f(c) + f(d) < f(a) + f(b)$. At the same time, the $f_2$-length of $[d,c]$ is less than that of $[a,b]$. This difference is at most the $f_2$ length of $[a,\infty)$. Since $f_1(c) = f_1(a) > f_2(a)$, we can find $e > c$ such that $[d,e]$ has the $f$-length $A$.  In addition, $f(c) + f(e) < f(c) + f(d) < f(a) + f(b)$, so that $[a,b]$ is not isoperimetric.

	\end{proof}
	
	\begin{proposition}
		\label{2.20}
		If the variance $a^2 \leq 1/2$, then the isoperimetric boundaries $B$ with one point $b$ in $[-c,c]$ cannot be of type \ref{2.14}\emph{(4)}.
	\end{proposition}
	
	\begin{proof}
		We may assume without loss of generality that $b \geq 0$. Suppose that $B$ is of type \ref{2.14}$(4)$. Then the region $L$ enclosed by $B$ consists of the union of an interval of type \ref{2.14}$(2)$ or \ref{2.14}$(3)$ and a ray. Apply Propositions \ref{2.17} and \ref{2.21} to get a new region $L'$ that beats the interval. Since $A < 1/2$, $L'$ may be chosen to not intersect the ray. Then the union of $L'$ and the ray beats $L$. 
	\end{proof}

	\begin{proposition}
		\label{2.22}
		If $B$ is an isoperimetric boundary and the variance $a^2 \leq 1/2$, then $B$ is a single point.
	\end{proposition}
	
	\begin{proof}
		If $B$ does not contain a point $s \in [-c,c]$, then by Proposition \ref{2.7}, then $B$ is a single point. Otherwise, apply Propositions \ref{2.17}-\ref{2.20} to complete the proof.
	\end{proof}
	
	\begin{proposition}
		\label{2.10}
		For the line endowed with density $f(x)$, if the variance $a^2$ is such that $1/2 \leq a^2 < 1$, then isoperimetric regions $R$ are always rays with boundary $B$ consisting of a single point.
	\end{proposition}
	
	\begin{proof}
		By Proposition \ref{3.36}, we have that $\psi '' (x) = 0$ exactly when $x$ is $c = \pm a^2\arccosh (1/a)$. Since $\psi ''' (x) > 0$ for $x < 0$ and $\psi ''' (x) < 0$ for $x > 0$, we have that $\psi ''$ is negative outside of $[-c,c]$ and is positive in $(-c,c)$.

		Suppose that $B$ is an isoperimetric boundary containing more than two points. By Corollary \ref{2.7}, $B$ does not lie entirely outside $[-c,c]$. Since $\psi '' (x) > 0$ on $(-c,c)$ and $\psi '' (\pm c) = 0$, the maximum and minimum of $\psi ' (x)$ on $[-c,c]$ are achieved at $c$ and $-c$ with $\psi ' (-c)$ negative and $\psi ' (c)$ positive. Since $\psi ' (x)$ tends to $-\infty$ as  $x$ approaches $\infty$, there exists a unique point $b>c$ such that $f(\pm b)= f(\pm c)$. Since $b > c$, $\psi '' (x) < 0$ outside of $[-b,b]$. 
		
		We claim that $B$ must lie in $[-b,b]$. Since $|\psi ' (x)|$ is constant on $B$, to show that $B \subset [-b,b]$ it suffices to show that the maximum and minimum of $\psi ' (x)$ on $[-b,b]$ are achieved at $-b$ and $b$. Since $0$ is a local minimum for $f(x)$, it suffices to show that $|\psi ' (b)| > |\psi ' (c)|.$ Since $\psi'(c)$ is postive and $\psi''(x) < 0$ for $x > c,$ there exists a unique point $d > c$ where $\psi'(d) = 0$ and $\psi'$ changes from positive to negative at $d$. To apply Lemma \ref{2.9}, consider functions $p$ and $q$ denoting the increase in $\psi$ moving left of $d$ and the decrease in $\psi$ moving right of $d$:
		$$p(x) = \psi(d) - \psi(d-x)$$
		$$q(x) = g(x) = \psi(d) - \psi(d+x)$$
		which satisfy the hypotheses of Lemma \ref{2.9}. Since $\psi(c) = \psi(b)$, we have that $$|\psi'(c)| = \psi'(c) = p'(d-c) < g'(b-d) -\psi'(b) = |\psi'(b)|.$$
		
		There are five candidates for the minimum points of $f(x)$ on $[-b,b]$: $\pm b, 0$, and $\pm d$. Since $d >c$, we have that $\psi''(d) < 0$ so that $\pm d$ is not a candidate. Since, also by the preceding paragraph, $\psi ' (x)$ is positive between $0$ and $c$, we have that $f(b) = f(c) > f(0)$. Therefore the minimum on this interval is $f(0)$. We have that $$\dfrac{d}{da}(f(0,a)) = -\dfrac{\sqrt{\frac{1}{a^2}}(-1+a^2)e^{-\frac{1}{2a^2}}}{a^3\sqrt{2\pi}} > 0$$ for all $a \in [-1/\sqrt{2}, 1)$. Therefore we have that $$2f(0,a) \geq 2f(0, 1/\sqrt{2}) \approx 0.415107...$$
		
		\begin{figure}[h]
			\includegraphics[height=2in]{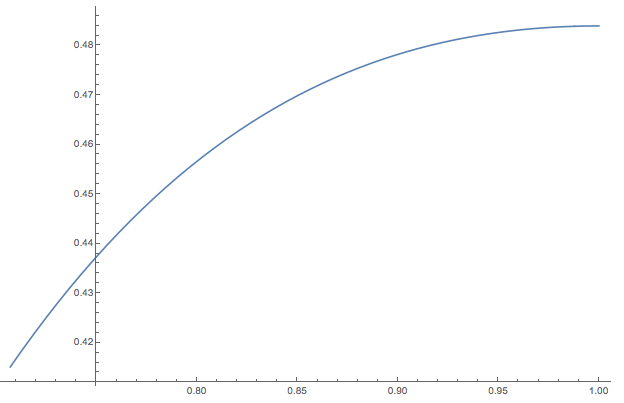}
			\caption{$2f(0,a)$ for various values of $a$}
			\label{fg2.6}
		\end{figure}
		
		To finish the proof, we must show that $f(x, a) < 0.415107....$ for all $x$ and all $a \in [1/\sqrt{2},1)$. Consider the numerator $n$ of $f$ given by $$n(x) = e^{-\frac{(x - 1)^2}{2a^2}} + e^{-\frac{(x + 1)^2}{2a^2}}.$$ We have that for a given $x$, $n$ increases when $a$ increases, so that $$n(x) \leq m(x) = e^{-\frac{(x - 1)^2}{2}} + e^{-\frac{(x + 1)^2}{2}}.$$ Since $$\dfrac{d}{dx}(\log{m(x)}) = \tanh (x) - x,$$ which has the same sign as $-x$, we see that $m(x)$ is maximized at $0$. Therefore $n(x) \leq m(0) < 1.22,$ so that $$f(x) < \dfrac{m(0)}{2\sqrt{2\pi}a} \leq \dfrac{1.22}{2\sqrt{\pi}} \approx 0.345.$$ This means that there is a ray which beats $B$, contradicting the fact that $B$ is isoperimetric.
	\end{proof}
	
	\begin{theorem}
		\label{2.18}
		The isoperimetric boundaries for the double Gaussian density $f$ are always single points enclosing rays.
	\end{theorem}
	
	\begin{proof}
		To cover the three cases, apply Propositions \ref{2.8}, \ref{2.22}, and \ref{2.10}.
	\end{proof}

	\section{Isoperimetric Regions on the Double Gaussian Plane}
	This section describes evidence for the conjecture of Ca\~{n}ete et al., stated here as Conjecture \ref{4.1}, which states that double Gaussian isoperimetric boundaries in the plane are vertical lines. Proposition \ref{4.5} proves that horizontal and vertical lines are the only stationary lines.  Proposition \ref{4.6} proves that vertical lines are better than horizontal lines. First we prove some incidental symmetry results (Propositions \ref{4.2} and \ref{4.3}).
	
	\begin{conjecture}[{\cite[Question 6]{canete}}]
		\label{4.1}
		Let $f(x,y) = e^{\psi(x,y)}$ be the normalized sum of two Gaussian densities with the same variance and different centers. Isoperimetric regions are half-planes enclosed by lines perpendicular to the line connecting the two centers.
	\end{conjecture}
	
	By the planar analogue of Proposition \ref{2.1}, it suffices to prove this conjecture in the case where the centers are $c_1 = (1,0)$ and $c_2 = (-1,0)$. 
	
	Then we have that $$f(x,y) = e^{\psi(x,y)} = \frac{1}{4\pi a^2}e^{-y^2/(2a^2)}(e^{-(x-1)^2/(2a^2)} + e^{-(x+1)^2/(2a^2)}).$$
	
	The next two propositions describe some symmetry properties of isoperimetric curves. For a curve $C$, let $A_C$ denote the weighted area enclosed by $C$.
	
	\begin{proposition}
		\label{4.2}
		Consider a density $g$ symmetric about the $x$-axis. If a closed, embedded curve $C$ encloses the same weighted area above and below the $x$-axis, then there is a curve $C'$ which is symmetric about the $x$-axis, encloses the same weighted area, and has weighted perimeter no greater than that of $C$. 
	\end{proposition}
	
	\begin{proof}
		Let $C_1$ and $C_2$ be the parts of $C$ in the open upper and lower half-planes chosen so that the weighted perimeter of $C_1$ is no bigger than that of $C_2$. Consider the curve $C'$ formed by joining $C_1$ with its reflection over the $x$-axis and taking the closure. Let $w$ denote the part of $C$ on the $x$-axis and $w_1$ denote the part of $C'$ on the $x$-axis. Since $g$ is symmetric about the $x$-axis, $A_C = A_{C'}$. In addition, $$|C'| - |C| = (2|C_1| + |w_1|) - (|C_1| + |C_2| + |w|) = (|C_1| - |C_2|) + (|w_1| - |w|).$$ We have that $|C_1| - |C_2| \leq 0$ by assumption, and since the part of $C$ which intersects the $x$-axis must include $w_1$, $|w_1| - |w| < 0$.  Therefore $|C'| -|C| \leq 0.$
	\end{proof}
	
	\begin{proposition}
		\label{4.3}
		Consider a density symmetric about the x-axis. If $C$ is a closed embedded planar curve symmetric about the x-axis, then the part $C'$ of $C$ in the open upper half-plane encloses half as much weighted area with half the weighted length.
	\end{proposition}
	
	\begin{proof}
		Suppose that $C$ is a curve that is symmetric about the $x$-axis and encloses area $A$. Since $C$ is symmetric about the $x$-axis, $C$ cannot have non-zero perimeter on the $x$-axis. Then $C'$ encloses area $A_C/2$ in the upper half-plane and has weighted perimeter $|C|/2$.  
	\end{proof}
	
	\begin{proposition}
		\label{4.5}
		If the plane is endowed with density $f$, then horizontal and vertical lines have generalized curvature 0 and are the only lines which have constant generalized curvature.
	\end{proposition}
	
	\begin{proof}
		Let $\psi = \ln{f}.$ Then $$\nabla\psi(x,y) = (\dfrac{-x+\tanh(x/a^2)}{a^2}, \dfrac{-y}{a^2}).$$ In addition, the normal to the line $y = cx + b$ is $(-c,1)/\sqrt{c^2 + 1}$ at all points of the line. Therefore the generalized curvature of such a line evaluated at $(0, b)$ is $$0-\nabla\psi(0,b)\cdot\frac{(-c,1)}{\sqrt{c^2 + 1}} = \frac{b}{a^2\sqrt{1 + c^2}},$$ and by an analogous computation the generalized curvature evaluted at $(1,c + b)$ is $$\dfrac{c+b}{a^2\sqrt{1+c^2}} + \dfrac{c(-1+\tanh(1/a^2))}{a^2\sqrt{1+c^2}}.$$ Thus the generalized curvature at $(0, b)$ and $(1, c+b)$ are equal exactly when $c = 0$. This shows that only non-vertical lines that could possibly have constant curvature are the horizontal lines $y = b$. Such lines have normal $(0,1)$, and this, combined for our formula with the gradient, shows that horizontal lines have constant curvature $b/a^2$.
		
		An explicit computation of the same variety shows that the vertical line $x = b$ has constant curvature $$\frac{b-\tanh(b/a^2)}{a^2}.$$ 
	\end{proof}
	
	\begin{proposition}
		\label{4.6}
		In the plane with double Gaussian density $f$, vertical lines enclose given area with less perimeter than horizontal lines.\end{proposition}
	
	\begin{figure}[h]
		\centering
		\begin{subfigure}{0.5\textwidth}
			\centering
			\includegraphics[width=\linewidth]{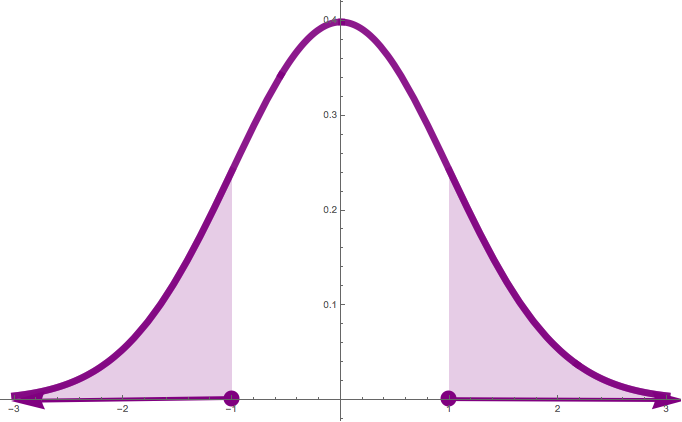}
			\caption{Symmetric Rays}
			\label{fig:sub1}
		\end{subfigure}%
		\begin{subfigure}{0.5\textwidth}
			\centering
			\includegraphics[width=\linewidth]{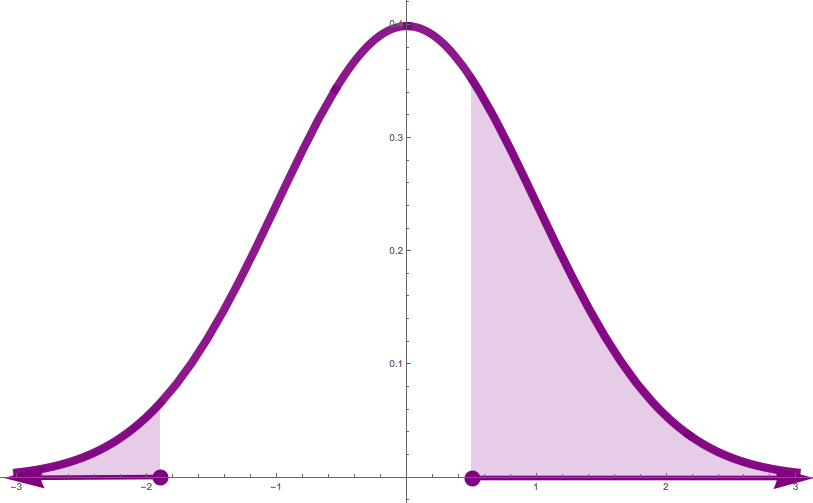}
			\caption{Unsymmetric Rays}
			\label{fig:sub2}
		\end{subfigure}
		\caption{When the purple areas are equal, the two unsymmetric rays are more efficient than the two symmetric rays. The efficiency increases as the disparity between the rays increases, and the limiting case is a single ray, which is the isoperimetric region.}
		\label{fg4.6}
	\end{figure}
	
	\begin{proof}
		We now compare the perimeters of and areas enclosed by the horizontal line $x = b$ and the vertical line $y = c$. By symmetry and the fact that we may assume the areas are less than $1/2$, we can assume that $b$ and $c$ are positive and consider the areas of the regions $x>b$ and $y > c$. 
		
		The area enclosed by the vertical line is $$\int_{b}^{\infty}\int_{-\infty}^{\infty} f(x,y) dydx = \int_{b}^{\infty} \dfrac{e^{-\frac{(x - 1)^2}{2a^2}} + e^{-\frac{(x + 1)^2}{2a^2}}}{2a\sqrt{2 \pi}},$$ which is the same as the weighted length of the ray $R_b = [b, \infty)$ on the double Gaussian line.
		The perimeter of the vertical line is 
		$$\int_{-\infty}^{\infty} f(b,y) dy =  \dfrac{e^{-\frac{(b - 1)^2}{2a^2}} + e^{-\frac{(b + 1)^2}{2a^2}}}{2a\sqrt{2 \pi}},$$ which is exactly the cost of $R_b$ on the double Gaussian line.

		The area enclosed by the horizontal line is $$\int_{c}^{\infty} \int_{-\infty}^{\infty} f(x,y) dxdy = \int_{c}^{\infty} \dfrac{e^{-\frac{y^2}{2a^2}}}{2a\sqrt{2 \pi}}dy,$$ which is the same as the weighted length of the ray $R_c = [c, \infty)$ on the single Gaussian (of total weighted-length 1) line.
		The perimeter of the horizontal line is $$\int_{-\infty}^{\infty} f(x,c) dx =  \dfrac{e^{-\frac{c^2}{2a^2}}}{2a\sqrt{2 \pi}},$$ which is exactly the cost of $R_c$ on the single Gaussian line.
		
		Therefore it suffices to show that a ray on the double Gaussian line of length $A$ costs less than a ray on the single Gaussian line of the same weighted length. Consider the line with density $g$  given by a single Gaussian of total length $1/2$. The ray on the single Gaussian is equivalent to the union of two disjoint, symmetric rays on the $g$-line. The ray on the double Gaussian is equivalent to the union of two disjoint, non-symmetric rays on the $g$-line. By applying the first and second variation arguments to a single Gaussian density, we see that two non-symmetric rays are always better than two symmetric rays of the same total weighted-length.
	\end{proof}

	Therefore if the isoperimetric curve corresponding to area $A$ is a line, then it is a vertical line. 
	
	\section*{ Acknowledgements}
	This paper is the work of the Williams College NSF ``SMALL'' 2015 Geometry Group. We thank our advisor Professor Morgan for his support. We would like to thank  the National Science Foundation, Williams College, the University of Chicago, and the MAA for supporting the ``SMALL'' REU and our travel to MathFest 2015.

	\bibliographystyle{abbrv}

	\bigskip
	
	\noindent John Berry
	\newline
	Department of Mathematics and Statistics
	\newline
	Williams College
	\newline
	jtb1@williams.edu
	\newline
	\newline
	Matthew Dannenberg
	\newline 
	Department of Mathematics
	\newline
	Harvey Mudd College
	\newline
	mdannenberg@g.hmc.edu
	\newline
	\newline
	Jason Liang
	\newline
	Department of Mathematics
	\newline
	University of Chicago
	\newline
	liangj@uchicago.edu
	\newline
	\newline
	Yingyi Zeng
	\newline
	Department of Mathematics and Computer Science
	\newline
	St. Mary's College of Maryland
	\newline
	yzeng@smcm.edu

\end{document}